\documentclass[11pt]{amsart}
\usepackage{amssymb, paralist, xspace, graphicx, url, amscd, euscript, mathrsfs,stmaryrd,epic,eepic,color}
\usepackage[all]{xy}
\usepackage{verbatim}
\SelectTips{xy}{12}
\SelectTips{cm}{}

\numberwithin{equation}{section}

1

\numberwithin{subsection}{section}

\allowdisplaybreaks[1]

\newcommand{\bbar}[1]{\setbox0=\hbox{$#1$}\dimen0=.2\ht0 \kern\dimen0 \overline{\kern-\dimen0 #1}}

\newtheorem*{namedtheorem}{\theoremname}
\newcommand{\theoremname}{testing}

\newtheorem{theorem}[subsection]{Theorem}
\newtheorem{proposition}[subsection]{Proposition}
\newtheorem{proposition-definition}[subsection]
{Proposition-Definition}
\newtheorem{corollary}[subsection]{Corollary}
\newtheorem{lemma}[subsection]{Lemma}

\theoremstyle{definition}
\newtheorem{definition}[subsection]{Definition}

\newtheorem{remark}[subsection]{Remark}

\theoremstyle{remark}

\DeclareMathOperator{\Ext}{Ext}

\newcommand\cC{\mathcal{C}}

\newcommand\cI{\mathcal{I}}

\newcommand\cK{\mathcal{K}}

\newcommand\cM{\mathcal{M}}

\newcommand\cO{\mathcal{O}}

\newcommand\ocM{\bbar{\mathcal{M}}}

\renewcommand\AA{\mathbb{A}}

\newcommand\GG{\mathbb{G}}

\newcommand\MM{\mathcal{M}}

\newcommand\NN{\mathbb{N}}

\newcommand\OO{\mathcal{O}}

\newcommand\PP{\mathbb{P}}
\newcommand\QQ{\mathbb{Q}}
\newcommand\RR{\mathbb{R}}

\newcommand\ZZ{\mathbb{Z}}

\newcommand\fM{\mathfrak{M}}

\newcommand\arr{\ifinner\to\else\longrightarrow\fi}

\def\displaytimes_#1{\mathrel{\mathop{\times}\limits_{#1}}}

\def\displayotimes_#1{\mathrel{\mathop{\bigotimes}\limits_{#1}}}

\newcommand\Spec{\operatorname{Spec}}

\newcommand{\proj}{\operatorname{Proj}}

\newdir{ >}{{}*!/-5pt/@{>}}

\newcommand\doublelong[2]{\mathbin{\xymatrix{{}\ar@<3pt>[r]^{#1}
\ar@<-3pt>[r]_{#2}&}}}

\newlength{\ignora}

\renewcommand{\setminus}{\smallsetminus}

\theoremstyle{plain}

\newtheorem{lem}[subsection]{Lemma}
\newtheorem{prop}[subsection]{Proposition}

\theoremstyle{definition}

\newtheorem{rem}[subsection]{Remark}

\numberwithin{equation}{subsection}

\newcommand{\lra}{\longrightarrow}

\newcommand{\mf}[1]{\mathfrak{#1}}
\DeclareMathOperator{\Trop}{Trop}

\begin{document}

\title[{Chow Quotients of Toric Varieties as Moduli of Stable Log Maps}]{Chow Quotients of Toric Varieties\\ as Moduli of Stable Log Maps}

\author{Qile Chen}

\author{Matthew Satriano}

\address{Department of Mathematics\\
Columbia University\\
New York, NY 10027}
\email{q\rule[-2pt]{.1cm}{0.5pt}chen@math.columbia.edu}

\address{Department of Mathematics, University of Michigan, 2074 East Hall, Ann Arbor, MI 48109}
\email{satriano@umich.edu}

\date{\today}
\begin{abstract}
Let $X$ be a projective normal toric variety and $T_0$ a rank one subtorus of the defining torus of $X$. We show that the normalization of the Chow quotient $X//T_0$, in the sense of Kapranov-Sturmfels-Zelevinsky, coarsely represents the moduli space of stable log maps to $X$ with discrete data given by $T_0\subset X$.
\end{abstract} 
\maketitle

\tableofcontents


\section{Introduction}
Throughout, we work over an algebraically closed field $k$ of characteristic 0. 

Chow quotients of toric varieties were introduced by Kapranov, Sturmfels, and Zelevinsky in \cite{tquot}.  Given a projective normal toric variety $X$ and a subtorus $T_0$ of the defining torus $T$, the \emph{Chow quotient} $X//T_0$ has the property that its normalization is the smallest toric variety which maps onto all GIT quotients of $X$ by $T_0$.  We show in this paper that when $T_0$ has rank one, the normalization of $X//T_0$ can be reinterpreted as the coarse moduli space of the stack of stable log maps, introduced by Abramovich and the first author \cite{DF1, DF2}, and independently by Gross and Siebert \cite{grosssiebert}.

We begin by recalling the construction of $X//T_0$.  For every point $x\in X$, the closure $Z_x:=\overline{T_0 x}$ of the orbit of $x$ under $T_0$ is a subvariety of $X$.  For $x\in T$, the orbit closures $Z_x$ have the same dimension and homology class.  We therefore obtain a morphism from $T':=T/T_0$ to the Chow variety $C(X)$ of algebraic cycles of the given dimension and homology class.  The Chow quotient $X//T_0$ is defined as the closure of $T'$ in $C(X)$.  It is a toric variety and the fan of its normalization is given explicitly in \cite[\S1]{tquot}.

Further assume now that $T_0$ is a rank one torus.  Let $Z_1$ be the closure of $T_0$ in $X$.  Then its normalization $\widetilde{Z}_1$ is isomorphic to $\PP^1$ and the induced morphism
\[
f_1:\PP^1\lra X
\]
can be viewed as a stable map with two marked points $\{0,\infty\}=\PP^1\setminus f_1^{-1}(T)$.  Let $\beta_0$ be the curve class of the stable map $f_1$ and let $c_0$ and $c_\infty$ be the contact orders of $0$ and $\infty$ with respect to the toric boundary $X\setminus T$. Roughly speaking, $c_0$ and $c_\infty$ are functions which assign to the marked points their orders of tangency with the components of $X\setminus T$ (see \cite{evspace} for more details). In the toric case, the contact orders can be explained as the slopes and weights of the unbounded edges of tropical curves associated to stable log maps, see Section \ref{ss:log-data}.

Our primary object of study in this paper is the stack $\cK_{\Gamma_0}(X)$ parameterizing stable log maps from rational curves with two marked points to $X$ such that the curve class is $\beta_0$ and the marked points have contact orders given by $c_0$ and $c_\infty$; here $\Gamma_0:=(0,\beta_0,2,\{c_0,c_\infty\})$ keeps track of the discrete data consisting of genus, curve class, number of marked points, and their tangency conditions.  Our main result is: 
\begin{theorem}
\label{thm:chowcs}
The normalization of $X//T_0$ is the coarse moduli space of $\cK_{\Gamma_0}(X)$.
\end{theorem}
\begin{remark}
\label{rmk:irred}
In particular, we see that $\cK_{\Gamma_0}(X)$ is irreducible.
\end{remark}
In the process of proving Theorem \ref{thm:chowcs}, we obtain an alternative description of $\cK_{\Gamma_0}(X)$ which is more akin to the construction of the Chow quotient.  As we saw above, $X//T_0$ is defined as the closure of $T':=T/T_0$ in the Chow variety $C(X)$.  Replacing $C(X)$ by other moduli spaces, we obtain alternate spaces birational to $X//T_0$.  Letting $Z_x$ be the orbit closure $\overline{T_0 x}$ as above, we see that for all $x\in T$, the normalization $\widetilde{Z}_x$ is isomorphic to $\PP^1$.  Thus, we obtain a stable map
\[
f_x:\PP^1\lra X
\]
with marked points $\{0,\infty\}=\PP^1\setminus f_x^{-1}(T)$.  These $f_x$ all have curve class $\beta_0$, and we obtain an immersion
\[
T'\lra \fM_{0,2}(X,\beta_0),
\]
where $\fM_{0,2}(X,\beta_0)$ denotes the Kontsevich space of stable maps to $X$ with genus 0, curve class $\beta_0$, and two marked points.  In analogy with the construction of the Chow variety, we let $\fM$ denote the closure of $T'$ in $\fM_{0,2}(X,\beta_0)$.  Then we have:
\begin{theorem}
\label{thm:comptwostacks}
$\cK_{\Gamma_0}(X)$ is the normalization of $\fM$.
\end{theorem}

\begin{remark}
There is an analogous picture if one assumes that $X$ is an affine normal toric variety and replaces $\fM_{0,2}(X,\beta_0)$ above by the toric Hilbert scheme, as defined in \cite{peevastillman}.  That is, for all $x\in T$, the $Z_x$ are $T'$-invariant closed subschemes of $X$ which have the same discrete invariants.  We therefore obtain an immersion from $T'$ to an appropriate toric Hilbert scheme.  The closure of $T'$ in this toric Hilbert scheme is called the main component.  In \cite[Thm 1.7]{loghilb}, Olsson shows that the normalization of the main component has a natural moduli interpretation in terms of log geometry.  Theorem \ref{thm:comptwostacks} above can therefore be viewed as an analogue of Olsson's theorem, replacing his use of the toric Hilbert scheme by the Kontsevich space.  That is, we show that the normalization of $\fM$ carries a moduli interpretation in terms of stable log maps.
\end{remark}


Recall that given any collection of discrete data $\Gamma=(g,\beta,n,\{c_i\}_{i=1}^n)$, it is shown in \cite{DF1, DF2, grosssiebert} that there is a proper Deligne-Mumford stack $\cK_\Gamma(X)$ which parameterizes stable log maps to $X$ from genus $g$ curves with $n$ marked points having curve class $\beta$ and contact orders given by the $c_i$.\footnote{Strictly speaking, \cite{DF1, DF2} only consider log schemes which are generalized Deligne-Faltings (see Definition \ref{def:genDF}), so to apply their theory, one must first show that the natural log structure on $X$ satisfies this hypothesis.  This is done in Proposition \ref{prop:genDF}, which we relegate to an appendix since the theory developed in \cite{grosssiebert} is already known to apply to toric varieties.} We show in Proposition \ref{prop:log-smooth} that if $g=0$, then $\cK_\Gamma(X)$ is log smooth, and in particular normal.  This is a key ingredient in the proof of Theorem \ref{thm:comptwostacks}, which we give in Section \ref{sec:logsm}.  In Section \ref{sec:tropcurves}, following \cite{NiSi, grosssiebert}, we explain the relationship between tropical curves and stable log maps to toric varieties.  While the use of tropical curves is not strictly necessary for this paper, they serve as a convenient tool to study the boundary of $\cK_\Gamma(X)$.  Theorem \ref{thm:chowcs} is then proved in Section \ref{sec:cs}.\\
\\
\noindent\textbf{Prerequisites:} We assume the reader is familiar with logarithmic geometry in the sense of Fontaine-Illusie-Kato 
(see for example \cite{kato} or \cite{logbook}).\\
\\
\noindent\textbf{Acknowledgments:} 
We would like to thank Dan Abramovich, Dustin Cartwright, Anton Geraschenko, Noah Giansiracusa, and Martin Olsson.  The first author was partially supported by the Simons Foundation. The second author was partially supported by NSF grant DMS-0943832 and an NSF postdoctoral fellowship (DMS-1103788).

\section{Log smoothness and irreducibility}
\label{sec:logsm}
Throughout this section, $X$ is a projective normal toric variety of dimension $d$ and $\Gamma$ is the discrete data $(0,\beta,n,\{c_i\})$.  Let $T$ be the defining torus of $X$ and $M$ be the character lattice of $T$.
\begin{proposition}\label{prop:log-smooth}
$(\cK_\Gamma(X),\MM_{\cK_\Gamma(X)})$ is log smooth over $(k,\OO_k^*)$.  Moreover, $\dim \cK_\Gamma(X) = \dim X + n-3$.
\end{proposition}
\begin{proof}
The universal curve on $\cK_\Gamma(X)$ induces a morphism of log stacks:
\[
\pi:(\cK_\Gamma(X),\MM_{\cK_\Gamma(X)})\lra (\fM_{0,n},\MM_{\fM_{0,n}}),
\]
where $(\fM_{g,n},\MM_{\fM_{g,n}})$ denotes the log stack of $(g,n)$-prestable curves; see \cite{fkato1} and \cite[Thm 1.10]{logcurve} for the definition and construction of this log stack.  
Since $(\fM_{g,n},\MM_{\fM_{g,n}})$ is log smooth over $(k,\OO_k^*)$, it suffices to show that $\pi$ is log smooth.  By \cite[Thm 4.6]{LogStack}, this is equivalent to showing that the induced morphism
\[
\pi':\cK_\Gamma(X)\lra\mathcal{L}og_{(\fM_{0,n},\MM_{\fM_{0,n}})}
\]
of stacks is smooth, where $\mathcal{L}og_{(S,\MM_S)}$ is the stack of log morphisms to a log scheme $(S,\MM_S)$, as defined in the introduction of (loc. cit.).\\
\\
Let $i:\Spec A\rightarrow\Spec A'$ be a square zero thickening of Artin local rings and let 
\[
\xymatrix{
\Spec A\ar[r]\ar[d]_{i} & \cK_{\Gamma}(X)\ar[d]^{\pi'}\\
\Spec A'\ar[r] & \mathcal{L}og_{(\fM_{0,n},\MM_{\fM_{0,n}})}
}
\]
be a commutative diagram.  We may view this as a commutative diagram of log stacks, by endowing the Artin local rings with the log structure pulled back from $\mathcal{L}og_{(\fM_{0,n},\MM_{\fM_{0,n}})}$. Hence the two vertical arrows are strict.  Denote the induced log structures on $\Spec A$ and $\Spec A'$ by $\MM_A$ and $\MM_{A'}$, respectively.  We therefore have a log smooth curve $h'$, a cartesian diagram
\[
\xymatrix{
(C,\MM_{C})\ar[r]\ar[d]_{h} & (C',\MM_{C'})\ar[d]^{h'}\\
(\Spec A,\MM_A)\ar[r] & (\Spec A',\MM_{A'})
}
\]
and a minimal stable log map $f:(C,\MM_C)\to (X,\MM_X)$, which we must show deforms to a minimal stable log map $f':(C',\MM_{C'})\to (X,\MM_X)$.
Since the minimality condition is open by \cite[Prop 3.5.2]{DF1}, it suffices to show that $f$ deforms as a morphism of log schemes.\\
\\
By standard arguments in deformation theory, it is enough to consider the case where the kernel $\cI$ of $A'\to A$ is principal and killed by the maximal ideal $\mathfrak{m}$ of $A'$.  Then the obstruction to deforming $f$ to a morphism of log schemes lies in
\[
\Ext^1(f_0^*\Omega^1_{(X,\MM_X)/k},\cO_{C_0})\otimes_{k} \cI
\]
where $f_0$ denotes the reduction of $f$ mod $\mathfrak{m}$, and $C_0$ denotes the fiber of $C$ over $A'/\mathfrak{m}=k$.  By \cite[Ex 5.6]{fkato2},
\[
\Omega^1_{(X,\MM_X)/k}\simeq\cO_X\otimes_\ZZ M.
\]
Therefore,
\[
\Ext^1(f_0^*\Omega^1_{(X,\MM_X)/k},\cO_{C_0})=H^1(\cO_{C_0}^d)=0
\]
where the last equality holds because $C_0$ is a curve of arithmetic genus $0$.  This shows that $(\cK_\Gamma(X),\MM_{\cK_\Gamma(X)})$ is log smooth.

To prove the claim about the dimension of $\cK_\Gamma(X)$, note that 
\[
\dim \Ext^0(f_0^*\Omega^1_{(X,\MM_X)/k},\cO_{C_0})= \dim H^{0}(\cO_{C_0}^{d}) = d,
\]
and so $\pi$ has relative dimension $d$.  Since $\dim \fM_{0,n} = n-3$, we see $\dim\cK_\Gamma(X)=d+n-3$.
\end{proof}

Let $\cK^{\circ}_{\Gamma}(X)$ denote the non-degeneracy locus, that is, the locus of $\cK_{\Gamma}(X)$ where the log structure $\MM_{\cK_\Gamma(X)}$ is trivial. By Proposition \ref{prop:log-smooth} and \cite[Prop 2.6]{niziol}, $\cK^{\circ}_{\Gamma}(X)$ is an open dense subset of $\cK_{\Gamma}(X)$. Consider the Kontsevich moduli space of stable maps $\fM_{0,n}(X,\beta)$. The forgetful map
\[
\Phi:\cK_{\Gamma}(X) \to \fM_{0,n}(X,\beta).
\]
sending a stable log map to its underlying stable map induces a locally closed immersion
\[
\cK_{\Gamma}^{\circ}(X) \to \fM_{0,n}(X,\beta).
\]
Let $\fM_{\Gamma}(X)$ be the closure of $\cK_{\Gamma}^{\circ}(X)$ in $\fM_{0,n}(X,\beta)$. Then $\Phi$ factors through a morphism
\[
\phi:\cK_{\Gamma}(X) \to \fM_{\Gamma}(X). 
\]
\begin{lemma}
\label{l:normalization}
$\phi$ is the normalization map.
\end{lemma}
\begin{proof}
By \cite[Corollary 3.10]{DF2} and Proposition \ref{prop:genDF}, the morphism $\Phi$ is representable and finite, and so $\phi$ is as well.  Since $(\cK_\Gamma(X),\MM_{\cK_\Gamma(X)})$ is fs and log smooth over $(k,\OO_k^*)$ by Proposition \ref{prop:log-smooth}, it follows that $\cK_\Gamma(X)$ is normal.  Since $\phi$ is an isomorphism over $\cK^\circ_\Gamma(X)$, it is birational, and so by Zariski's Main Theorem, $\phi$ is the normalization map.
\end{proof}
For the rest of this section, we return to the setting and notation of the introduction, and let $\Gamma=\Gamma_0$.  Just as $X//T_0$ (resp $\fM$) is constructed by taking the closure of $T'$ in the Chow variety (resp the Kontsevich space), we can perform a similar construction with $\cK_\Gamma(X)$.  Namely, for $x\in T$, the stable map
\[
f_x: \PP^1\to X
\]
is naturally a stable log map.  We therefore obtain a morphism $T'\to \cK_\Gamma(X)$.  Let $\mf{X}_\Gamma$ denote the closure of $T'$ in $\cK_\Gamma(X)$.  
The forgetful morphism $\Phi$ 
then induces a map
\[
\phi':\mf{X}_\Gamma\lra \fM.
\]
\begin{lemma}
\label{l:XGammaopen}
$\mf{X}_\Gamma$ is an open substack of $\cK_\Gamma(X)$, and so $\phi'$ is the normalization map.
\end{lemma}
\begin{proof}
As in the proof of Lemma \ref{l:normalization}, $\phi'$ is representable and finite.  If $\mf{X}_\Gamma$ is an open substack of $\cK_\Gamma(X)$, it is then normal.  Since $\phi'$ is an isomorphism over $T'$, 
Zariski's Main Theorem shows that it is the normalization map.\\
\\
To show that $\mf{X}_\Gamma$ is open in $\cK_\Gamma(X)$, it suffices to prove that $\cK^{\circ}_\Gamma(X)$ and $\mf{X}^{\circ}_\Gamma:=\mf{X}_\Gamma\cap\cK^{\circ}_\Gamma(X)$ have the same dimension.  Since $T'$ is dense in $\mf{X}_\Gamma$, we see that $\mf{X}_\Gamma$ has dimension $d-1$.  On the other hand, the map
\[
\pi:(\cK_\Gamma(X),\MM_{\cK_\Gamma(X)})\lra (\fM_{0,2},\MM_{\fM_{0,2}})
\]
in the proof of Proposition \ref{prop:log-smooth} induces a map
\[
\cK^{\circ}_\Gamma(X)\lra\fM^{\circ}_{0,2},
\]
where $\fM^{\circ}_{0,2}$ denotes the open substack of $\fM_{0,2}$ with smooth fiber curves.  By Proposition \ref{prop:log-smooth}, we see that $\cK^{\circ}_{\Gamma}(X)$ has dimension $d-1$.
\end{proof}
Since $\phi'$ is the normalization map, to prove Theorem \ref{thm:comptwostacks}, we must show $\mf{X}_\Gamma = \cK_\Gamma(X)$.  Since $\mf{X}_\Gamma$ is an open and closed substack of $\cK_\Gamma(X)$, the following proposition suffices.
\begin{prop}
\label{prop:irred}
$\cK_{\Gamma}(X)$ is irreducible.  
\end{prop}
\begin{proof}
It is enough to prove that $\cK^{\circ}_{\Gamma}(X)$ is irreducible.  Let $s\in\cK^{\circ}_{\Gamma}(X)(k)$, and $f:\mathbb{P}^1\to X$ be the stable log map corresponding to $s$.  Note that the log structure of the boundary of $X$ is everywhere non-trivial. Since the log structure is trivial at $s$, the image of $f$ necessarily meets $T$.  To prove that $\cK^{\circ}_{\Gamma}(X)$ is irreducible, it is enough to show that we can act on $f$ by an element of $T$ to obtain a map isomorphic to $f_1$ from the introduction.

After acting on $f$ by some element of $T$, we may assume that $f$ sends $1\in\PP^1$ to 
$1\in T\subset X$.  Choose a maximal cone $\sigma$ in the fan of $X$ such that the associated affine open toric variety $U\subset X$ contains $f(0)$.  Restricting $f$ to $U$, we obtain a map $f': V=\Spec k[t] \lra U$.

Let $P$ be the monoid $\sigma^{\vee}\cap M$ and let $e_1,\dots,e_\ell$ be the irreducible elements of $P$.  We see that for each $i$, 
\[
f^*(e_i)=t^{c_i}a_i,
\]
where $c_i$ is the contact order prescribed by $\Gamma$ and $a_i$ is some element of $k[t]$.  Note that if $\alpha\in k$ is a root of $a_i$, then the point $t=\alpha$ is mapped to the toric boundary; however, the contact order given by $\Gamma$ implies that $t=0$ is the only point in $V$ which maps to the boundary.  Hence, $a_i$ must be a power of $t$.  But if $a_i$ is divisible by $t$, then the contact order of $t=0$ along $e_i=0$ is greater than $c_i$.  Therefore, $a_i$ must be a non-zero constant.

Now observe that the point $1\in T\subset U$ is given by $e_i=1$ for all $i$.  Since $f(1)=1$, the equation $f^*(e_i)=t^{c_i}a_i$ shows that $a_i=1$.  This shows that $f$ is uniquely determined over $U$.  Since $f_1$ also satisfies these constraints, we see that $f$ and $f_1$ agree over $U$.  Since $f$ and $f_1$ are two maps from $\PP^1$ to $X$ which agree on a dense open subset of the source, they are equal.
\end{proof}

\section{Tropical curves associated to stable log maps}
\label{sec:tropcurves}
The goal of this section is to prove Proposition \ref{prop:chain-ofP1s}.  Following \cite{NiSi, grosssiebert}, we explain the connection between tropical curves and stable log maps to toric varieties.

\subsection{Review of tropical curves}
Let $\overline{G}$ be the geometric realization of a weighted, connected finite graph with weight function $\omega$.  That is, $\overline{G}$ is the CW complex associated to a finite connected graph with vertex set $\overline{G}^{[0]}$ and edge set $\overline{G}^{[1]}$, and 
\[
\omega:\overline{G}^{[1]}\to \NN
\]
is a function.  Here we allow $\overline{G}$ to have divalent vertices.  Given an edge $l \in \overline{G}^{[1]}$, we denote its set of adjacent vertices by $\partial l$. If $l$ is a loop, then we require $\omega(l)=0$.

Let $G^{[0]}_{\infty}\subset \overline{G}^{[0]}$ be the set of one-valent vertices, and let 
\[
G := \overline{G} \setminus \overline{G}^{[0]}_{\infty}.
\]
Let $G^{[1]}_{\infty}$ be the set of non-compact edges in $G$, which we refer to as {\em unbounded edges}. A {\em flag} of $G$ is a pair $(v,l)$ where $l$ is an edge and $v\in\partial l$.  We let $FG$ be the set of flags of $G$, and for each vertex $v$, we let
\[FG(v) :=  \{ (v,l) \in FG\}.\]
Let $N$ be a lattice and $M = N^{\vee}$. We let $N_{\QQ}:=N\otimes_{\ZZ}\QQ$ and $N_{\RR}:=N\otimes_{\ZZ}\RR$.

\begin{definition}\label{def:tropical-curve}
A {\em parameterized tropical curve} in $N_{\QQ}$ is a proper map $\varphi: G \to N_{\RR}$ of topological spaces satisfying the following conditions:
\begin{enumerate}
\item For every edge $l$ of $G$, the restriction $\varphi|_{l}$ acts as dilation by a factor $\omega(l)$ with image $\varphi(l)$ contained in an affine line with rational slope. If $\omega(l) = 0$, then  $\varphi(l)$ is a point.
 
\item For every vertex $v$ of $G$, we have $\varphi(v) \in N_{\QQ}$. 

\item For each $(v,l) \in FG(v)$, let $u_{v,l}$ be an primitive integral vector emanating from $\varphi(v)$ along the direction of $h(l)$. Then 
\[
\epsilon_v:=\sum_{(v,l)\in FG(v)}\omega(l)u_{v,l} = 0,
\]
which we refer to as the {\em balancing condition}.
\end{enumerate}
\end{definition}

An {\em isomorphism} of tropical curves $\varphi: G \to N_{\RR}$ and $\varphi' : G' \to N_{\RR}$ is a homeomorphism $\Phi: G \to G'$ compatible with the weights of the edges such that $\varphi = \varphi'\circ \Phi$.

A {\em tropical curve} is an isomorphism class of parameterized tropical curves.

\subsection{Tropical curves from non-degenerate stable log maps}\label{ss:log-data}
Let $(X,\MM_X)$ be a toric variety with its standard log structure, and let $T \subset X$ be its defining torus. We denote by $N$ the lattice of one-parameter subgroups of $T$.  Let $f: (C,\MM_C) \to (X,\MM_X)$ be a stable log map over $(S,\MM_S)$ with $S$ a geometric point. Further assume that $f$ is non-degenerate; that is, the log structure $\cM_{S}$ is trivial. 

In this subsection, we show how to assign a tropical curve $\Trop(f): G\to N_\RR$ to any such non-degenerate stable log map $f$.  To begin, let $G$ be the graph with a single vertex $v$, which we think of as being associated to the unique component of $C$, and with one unbounded edge for each marked point of $C$. We let $\Trop(f)(v)=0$.

Let $l$ be an edge corresponding to a marked point $p$ of $C$.  If $p$ has trivial contact orders, then we set $\omega(l) = 0$ and let $\Trop(f)$ contract $l$ to $0$.  Otherwise, the contact order is equivalent to giving a non-trivial map
\[
c_{l}: \ocM_{X,f(p)} \to \bbar{\MM}_{C,p}=\NN.
\]
Note that we have a surjective cospecialization map of groups
\[
M:=N^{\vee} \to \ocM^{gp}_{X,f(p)}
\]
corresponding to the specialization of the generic point of $T$ to $f(p)$.  Composing with $c_l^{gp}$, we obtain a map 
\[
\mu_{l}: M \to \ZZ,
\]
which defines an element $\mu_{l} \in N$. Let  $u_{l}$ be the primitive vector with slope given by $\mu_{l} \in N$. We define $\omega(l)$ to be the positive integer such that $\mu_{l} = \omega(l) u_{l}$, and define the image $\Trop(f)(l)$ to be the unbounded ray emanating from $0$ along the direction of $u_{l}$. This defines our desired map $\Trop(f): G \to N_{\RR}$ up to reparameterization.
\begin{proposition}
\label{prop:nondegentropcurve}
$\Trop(f): G \to N_{\RR}$ defines a tropical curve.
\end{proposition}
\begin{proof}
It remains to check that the balancing condition holds.  That is, we must show $\epsilon_v=0$.  Note that every $m\in M$ defines a rational function on $C$ and that the degree of the associated Cartier divisor is $0=\epsilon_v(m)$.  Therefore, $\epsilon_v\in N=M^\vee$ is $0$.
\end{proof}


\subsection{Tropical curves from stable log maps over the standard log point}
Let $(X,\MM_X)$ be a toric variety with its standard log structure, and let $T\subset X$ be its defining torus.  Fix discrete data $\Gamma=(g,\beta,n,\{c_i\})$ and let $f: (C,\MM_C) \to (X,\MM_X)$ be a stable log map with discrete data $\Gamma$ over the standard log point $(S,\MM_S)$; that is, $S$ is a geometric point and $\MM_S$ is the log structure associated to the map $\NN\to\cO_S$ sending $1$ to $0$. This is equivalent to giving a (not necessarily strict) log map 
\[(S,\cM_{S}) \to (\cK_{\Gamma}(X), \cM_{\cK_{\Gamma}(X)}),\]
and the stable log map $f$ is obtained by pulling back the universal stable log map over $(\cK_{\Gamma}(X), \cM_{\cK_{\Gamma}(X)})$.  In this subsection, we associate a tropical curve 
\[
\Trop(f):G\to N_\RR
\]
to $f$ by modifying the construction given in \cite[\S1.3]{grosssiebert}.

We define $G$ to be the dual graph of $C$ where we attach an unbounded edge for each marked point.  Given a vertex $v$, let $t$ be the generic point of the corresponding component of $C$.  We therefore have a morphism
\[
\ocM_{X,f(t)} \to \ocM_{C,t}=\NN
\]
of monoids.  Taking the associated groups and composing with the cospecialization map $M \to \bbar{\cM}_{X,f(t)}^{gp}$ yields a map
\[
\tau_v:M\to \ZZ,
\]
and hence a point in $N$.  We define $\Trop(f)(v)=\tau_v$.

Let $l$ be an edge of $G$.  If $\partial l=\{v,v'\}$ and $v\neq v'$, then we define the image of $l$ under $\Trop(f)$ to be the line segment joining $\tau_v$ and $\tau_{v'}$.  In this case, $\tau_{v'}-\tau_v=e_l\mu_l$, where $e_l\in\bbar{\MM}_S=\NN$ is the section which smooths the node corresponding to $l$, and $\mu_l$ is an element of $N$.  We define $\omega(l)$ to be the positive integer such that $\mu_l=\omega(l)u_l$, where $u_l$ is a primitive integral vector.

Suppose now that $l$ is an unbounded edge corresponding to a marked point $p$.  If $p$ has trivial contact orders, then we set $\omega(l)=0$ and let $\Trop(f)$ contract $l$ to $\tau_v$, where $\partial l=\{v\}$.  Otherwise, the contact orders of $p$ define a non-trivial map
\[
c_l:\bbar{\MM}_{X,f(p)}\to\bbar{\MM}_{C,p}=\NN\oplus\bbar{\MM}_S\to \NN,
\]
where the last map is the projection.  Again taking the associated groups and composing with the cospecialization map $M \to \bbar{\cM}^{gp}_{X,f(p)}$, we obtain
\[
\mu_{l}: M \to \ZZ.
\]
We define $\omega(l)$ to be the positive integer such that $\mu_l=\omega(l)u_l$, where $u_l\in N$ is a primitive integral vector, and we let $\Trop(f)(l)$ be the unbounded ray emanating from $\tau_v$ in the direction of $u_l$.

\begin{prop}\label{prop:map-trop-curve}
$\Trop(f): G \to N_{\RR}$ defines a tropical curve.
\end{prop}
\begin{proof}
We must check that the balancing condition holds for each vertex $v$ of $G$.  As in the proof of Proposition \ref{prop:nondegentropcurve}, every $m\in M$ defines a rational function on the irreducible component of $C$ corresponding to $v$.  The degree of the associated Cartier divisor is $0=\epsilon_v(m)$, and so $\epsilon_v=0$, \emph{c.f.} \cite[Proposition 1.14]{grosssiebert}.
\end{proof}

\begin{remark}
\label{rmk:specialization}
Let $R$ be the complete local ring of $\AA^1$ at the origin, and let $\MM_R$ be the log structure on $R$ induced by the standard log structure on $\AA^1$.  Denote the closed and generic points of $\Spec R$ by $0$ and $\eta$, respectively.  Suppose $h:(\cC,\cM_{\cC})\to (X,\cM_X)$ is a stable log map over $R$ with discrete data $\Gamma$ such that $h_0=f$.  Note that $h_\eta$ is a non-degenerate stable log map.  For each marked section $p:\Spec R\to \cC$, let $l_0$ and $l_\eta$ be the edges of the dual graphs of $\cC_0$ and $\cC_\eta$ corresponding to the marked points $p_0$ and $p_\eta$, respectively.  Consider the morphism
\[
\bbar{\cM}_{X}|_{h(p)}\to\bbar{\cM}_{\cC}|_p=\NN\oplus\bbar{\cM}_R\to\NN,
\]
where the last map is the projection.  Taking associated groups and precomposing with the map $M\to \bbar{\cM}^{gp}_{X}|_{h(p)}$, we obtain a map $M\to\ZZ$ of constant sheaves on $\Spec R$ whose special and generic fibers are $\mu_{l_0}$ and $\mu_{l_\eta}$.  Hence, we see $\mu_{l_0}=\mu_{l_\eta}$.
\end{remark}

The following result plays an important role in the proof of Theorem \ref{thm:chowcs}.

\begin{proposition}
\label{prop:chain-ofP1s}
If the discrete data $\Gamma$ is given by $g=0$, $n=2$, and $\beta \neq 0$, then $\Trop(f)$ is an embedding whose image is a line.  Moreover, $C$ is a chain of $\PP^1$s and $f$ does not contract any components of $C$.
\end{proposition}
\begin{proof}
Since $\cK_\Gamma(X)$ is log smooth by Proposition \ref{prop:log-smooth}, there exists a stable log map $h:(\cC,\cM_{\cC})\to (X,\cM_X)$ over $(R,\cM_R)$ as in Remark \ref{rmk:specialization}.  Let $p,p':\Spec R\to\cC$ be the two marked sections, and let $l_0$, $l'_0$, $l_\eta$, $l'_\eta$ be the corresponding edges of the dual graphs of $C$ and $\cC_\eta$.  Since $\beta\neq0$, the two marked points $p_\eta$ and $p'_\eta$ of $\cC_\eta$ have non-trivial contact orders.  The balancing condition for $\Trop(h_\eta)$ then shows $\mu_{l'_\eta}=-\mu_{l_\eta}\neq0$.  By Remark \ref{rmk:specialization}, we therefore have $\mu_{l'_0}=-\mu_{l_0}\neq0$.  In particular, $\Trop(f)$ maps $l_0$ and $l'_0$ to unbounded rays.

We next show that if $l$ is an edge of $G$, then $\Trop(f)(l)$ is a point, or it is a line segment or ray parallel to $\mu_{l_0}$.  Suppose $\Trop(f)(l)$ is not a point.  If $\Trop(f)(l)$ is unbounded, then $l$ is $l_0$ or $l'_0$, and so $\Trop(f)(l)$ is parallel to $\mu_{l_0}$.  Otherwise, $\Trop(f)(l)$ is a line segment and $\partial l=\{v,v_1\}$ with $v\neq v_1$.  If $\Trop(f)(l)$ is not parallel to $\mu_{l_0}$, then the balancing condition shows that there is an edge $l_1\neq l$ such that $v_1\in \partial l_1$ and $\Trop(f)(l_1)$ is not parallel to $\mu_{l_0}$.  Hence, $l_1$ is a line segment with endpoints $v_1$ and $v_2$.  Again, the balancing condition shows that there is an edge $l_2$ containing $v_2$ such that $\Trop(f)(l_2)$ is a line segment which is not parallel to $\mu_{l_0}$.  Since $C$ has genus $0$, we see $l$, $l_1$, and $l_2$ are distinct.  Continuing in this manner, we produce an infinite sequence of distinct edges $l_i$ of the dual graph of $C$.  This is a contradiction.

Lastly, we show that every irreducible component $A$ of $C$ has exactly two special points.  Hence, $C$ is a chain of $\PP^1$s, $f$ does not contract any component of $C$, and $\Trop(f)(G)$ is a line parallel to $\mu_{l_0}$.  Suppose $A$ is a component with at least three special points and let $v$ be the vertex of $G$ corresponding to $A$.  Then $G\setminus v$ is a disjoint union of non-empty trees $T_1,T_2,\dots,T_m$ with $m\geq3$.  Without loss of generality, $T_1$ only contains bounded edges.  The argument in the preceding paragraph then shows that $\Trop(f)$ maps every edge of $T_1$ to a single point.  If $C_1$ denotes the subcurve of $C$ corresponding to $T_1$, then we see that every special point of $C_1$ has a trivial contact order, and so $f$ contracts $C_1$.  Since $T_1$ is a tree, $C_1$ contains components with only two special points.  This contradicts the stability of $f$.
\end{proof}

\section{The Chow quotient as the coarse moduli space}
\label{sec:cs}

Throughout this section, we let $\Gamma=\Gamma_0$ and $C(X)$ denote the Chow variety as in the introduction.  Let $K$ be the normalization of $X//T_0$.  Note that there is a map
\[
F:\cK_\Gamma(X)\longrightarrow C(X)
\]
sending a stable log map $f:(C,\MM_C)\to (X,\MM_X)$ to the image cycle $f_*[C]$.  Since $\cK_\Gamma(X)$ is irreducible by Theorem \ref{thm:comptwostacks}, $F$ factors as
\[
\cK_\Gamma(X)\stackrel{F'}{\longrightarrow} X//T_0\stackrel{i}{\longrightarrow} C(X),
\]
where $i$ is the natural inclusion.  Since $F$ is an isomorphism over $T'$ and $\cK_\Gamma(X)$ is normal, by Proposition \ref{prop:log-smooth}, we obtain an induced morphism
\[
G:\cK_\Gamma(X)\longrightarrow K
\]
To prove Theorem \ref{thm:chowcs}, we show
\begin{proposition}
\label{prop:coarse}
$G$ is a coarse space morphism.
\end{proposition}
\begin{proof}
Since both $\cK_\Gamma(X)$ and $K$ are normal and proper, and since $G$ is an isomorphism over $T'$, by Zariski's Main Theorem, it suffices to show $G$ is quasi-finite.  To do so, it is enough to show $F'$ is quasi-finite at the level of closed points.  That is, we show that if $x\in X//T_{0}$ is a closed point and $E_{x}$ denotes the corresponding cycle of $X$, then there are finitely many stable log maps whose image cycles are given by $E_{x}$. Let
\[
E_{x} = \sum a_{i}Z_{i},
\]
where the $a_{i}$ are positive integers and the $Z_{i}$ are reduced irreducible closed subschemes of $X$.  
Let $\widetilde{Z}_{i}$ be the normalization of $Z_{i}$. Since $E_{x}$ is of dimension $1$, we have $\widetilde{Z}_{i}\simeq \PP^{1}$.

We claim that if $f:(C,\MM_C)\to (X,\MM_X)$ is a stable log map that defines a closed point of $\cK_\Gamma(X)$ such that the image cycle of $f$ is $E_x$, then 
$f$ can only be ramified at the special points of $C$.  Given this claim, $F'$ is quasi-finite.  Indeed, since Proposition \ref{prop:chain-ofP1s} shows that no component of $C$ is contracted under $f$, the number of irreducible components of $C$ is bounded by $\sum a_{i}$. For each irreducible component $A$ of $C$, the map $f|_A$ factors as
\[
A \longrightarrow \widetilde{Z}_{i} \longrightarrow X
\]
for some $i$.  Since the first map $A \to \widetilde{Z}_{i}$ can only be ramified at the two fixed special points, it is determined by the degree of $f|_A$.  Thus, there are only finitely many choices for $f$.

It remains to prove the claim. By Proposition \ref{prop:irred}, $\cK_\Gamma(X)$ is irreducible and $T'$ is dense, so there exists a toric morphism $\AA^1\to\cK_\Gamma(X)$ such that the fiber over $0\in\AA^1$ is our given stable log map $f:(C,\MM_C)\to (X,\MM_X)$ whose image cycle is $E_x$.  Let $R$ denote the complete local ring $\widehat{\cO}_{\AA^1,0}$ and let
\[
\xymatrix{
\cC\ar[r]^h\ar[d] & X\\
\Spec R & 
}
\]
be the associated stable map.  Let $\eta\in\Spec R$ be the generic point.

We first handle the case when $X$ is smooth. Let $\cC^{\circ}$ be the open subset of $\cC$ obtained by removing the special points. Note that $\cC^{\circ}$ is normal, and $h|_{\cC^{\circ}}$ is quasi-finite by Proposition \ref{prop:chain-ofP1s}. By the purity of the branch locus theorem \cite[p.461]{altmankleiman}, if $h|_{\cC^{\circ}}$ is ramified, then the ramification locus $D$ is pure of codimension 1.  Since $h|_{\cC^{\circ}}$ is not everywhere ramified over the central fiber, $D$ must intersect the generic fiber.  However, $h|_{\cC^{\circ}}$ is an embedding over the generic fiber, so we conclude that $D$ is empty.

We now consider the case when $X$ is singular.  Let $p: \widetilde{X}\to X$ be a toric resolution.  We may replace $R$ by a ramified extension, as this does not affect the set of closed points.  Since the natural map $\cK_{\Gamma}(\widetilde{X})\to \cK_\Gamma(X)$ is proper, by the valuative criterion, we can assume we have a stable map $\widetilde{h}:\widetilde{\cC}\to \widetilde{X}$ and a commutative diagram
\[
\xymatrix{
\widetilde{\cC} \ar[r]^{\widetilde{h}} \ar[d]_q & \widetilde{X}\ar[d]^{p}\\
\cC\ar[r]^h  & X
}
\]
over $R$.  Here $h$ is obtained by taking the stabilization of the prestable map $p\circ\widetilde{h}$.  The previous paragraph shows that $\widetilde{h}$ only ramifies at the special points.  Since Proposition \ref{prop:chain-ofP1s} shows that $\widetilde{\cC}$ and $\cC$ are both chains of $\PP^1$s, we see that $h$ only ramifies at the special points as well.
\end{proof}

{\section*{Appendix A: Toric varieties have generalized Deligne-Faltings log structures}
\renewcommand{\thesection}{A}
\refstepcounter{section}
\label{sec:genDF}
The theory of moduli spaces of stable log maps $\cK_\Gamma(Y,\MM_Y)$ is developed in \cite{DF1,DF2} and \cite{grosssiebert} for different classes of log schemes $(Y,\MM_Y)$.  In \cite{DF1,DF2}, Abramovich and the first author consider log schemes 
which are generalized Deligne-Faltings 
(see Definition \ref{def:genDF}); in \cite{grosssiebert}, Gross and Siebert 
consider log schemes which are quasi-generated Zariski. It is shown in \cite[Prop 4.8]{DF2} that when $(Y,\cM_{Y})$ is both generalized Deligne-Faltings and quasi-generated Zariski, the Abramovich-Chen and Gross-Siebert constructions are identical. Gross-Siebert show that the standard log structure $\MM_X$ on a normal toric variety $X$ is always quasi-generated Zariski.  Here we show that if $X$ is also projective, then $\MM_X$ is generalized Deligne-Faltings.  Therefore, the two theories agree for projective normal toric varieties.
\begin{definition}
\label{def:genDF}
A log structure $\MM_Y$ on a scheme $Y$ is called \emph{generalized Deligne-Faltings} if there exists a fine saturated sharp monoid $P$ and a morphism $P\to\bbar{\MM}_Y$ which locally lifts to a chart $P\to\MM_Y$.
\end{definition}
\begin{remark}
\label{rmk:genDF}
Given a fine saturated sharp monoid $P$, let $A_P=\Spec k[P]$ with its standard log structure $\MM_{A_P}$.  Then there is a natural action of $T_P:=\Spec k[P^{gp}]$ on $(A_P,\MM_{A_P})$ induced by the morphism $P\to P\oplus P^{gp}$ sending $p$ to $(p,p)$.  The log structure $\MM_{A_P}$ descends to yield a log structure $\MM_{[A_P/T_P]}$ on the quotient stack $[A_P/T_P]$.  By \cite[Rmk 5.15]{LogStack}, a log scheme $(Y,\MM_Y)$ is generalized Deligne-Faltings if and only if there exists a strict morphism
\[
(Y,\MM_Y)\lra ([A_P/T_P],\MM_{[A_P/T_P]})
\]
for some fine saturated sharp monoid $P$.
\end{remark}
Let $X$ be a projective normal toric variety and let $\MM_X$ be its standard log structure.  Let $Q\subset \RR^n$ be a polytope associated to a sufficiently positive projective embedding of $X$.  Placing $Q$ at height 1 in $\RR^n\times\RR$ and letting $P$ be the monoid of lattice points in the cone over $Q$, we have $X=\proj k[P]$.  Note that $P$ is fine, saturated, and sharp.  
Let $(A_P,\MM_{A_P})$ be as in Remark \ref{rmk:genDF}, let $U$ be the compliment of the closed subscheme of $A_P$ defined by the irrelevant ideal of $k[P]$, and let $\MM_U=\MM_{A_P}|_U$.  The function $\deg:P\to\ZZ$ sending an element to its height induces a $\GG_m$-action on $(A_P,\MM_{A_P})$.  Hence, $\MM_U$ descends to yield a log structure $\MM_P$ on $X$.
\begin{lemma}
\label{l:MPgenDF}
$\MM_P$ is generalized Deligne-Faltings.
\end{lemma}
\begin{proof}
We have a cartesian diagram
\[
\xymatrix{
(U,\MM_U)\ar[r]\ar[d] & (A_P,\MM_P)\ar[d]\\
(X,\MM_P)\ar[r] & ([A_P/\GG_m],\MM_{[A_P/\GG_m]})
}
\]
where all morphisms are strict and the vertical morphisms are smooth covers.  Note that the $\GG_m$-action on $(A_P,\MM_{A_P})$ is induced from the morphism $\sigma:P\to P\oplus\ZZ$ defined by $p\mapsto (p,\deg p)$.  Since $\sigma$ factors as
\[
P\lra P\oplus P^{gp}\lra P\oplus\ZZ
\]
where the first map is $p\mapsto(p,p)$ and the second is $(p,\xi)\mapsto (p,\deg \xi)$, we see that there is a strict smooth cover
\[
([A_P/\GG_m],\MM_{[A_P/\GG_m]})\lra ([A_P/T_P],\MM_{[A_P/T_P]}).
\]
Hence, Remark \ref{rmk:genDF} shows that $\MM_P$ is generalized Deligne-Faltings.
\end{proof}
Note that $\MM_P|_T=\cO_T^*$, where $T$ is the torus of $X$.  We therefore obtain a map
\[
\psi:\MM_P\lra j_*^{log}\cO_T^*=:\MM_X.
\]
\begin{proposition}
\label{prop:genDF}
$\psi$ is an isomorphism, and so $(X,\MM_X)$ is generalized Deligne-Faltings.
\end{proposition}
\begin{proof}
To show $\psi$ is an isomorphism, it is enough to look Zariski locally on $X$.  Note that $X$ has an open cover by the $X_v:=\Spec k[Q_v]$, where $v$ is a vertex of the polytope $Q$ and $Q_v$ is the monoid of lattice points in the cone over $Q-v:=\{q-v\mid q\in Q\subset\RR^n\}$.  Let $P_v$ be the submonoid of $P^{gp}$ generated by $P$ and $-v$.  Then we have a cartesian diagram
\[
\xymatrix{
A_{P_v}\ar[r]^i\ar[d]_\pi & U\ar[d]\\
X_v\ar[r] & X
}
\]
where $\pi$ is induced from the map $Q_v\to P_v$ embedding $Q_v$ at height 0 in $P_v$, and where the composite of $i$ and $U\to A_P$ is induced from the inclusion $P\to P_v$.  Hence,
\[
\MM_{Q_v}=(\MM_{P_v})^{\GG_m}
\]
and so $\psi$ is an isomorphism over $X_v$.
\end{proof}




\end{document}